\documentclass[11pt]{article}

\usepackage[T1]{fontenc}
\usepackage[utf8]{inputenc}
\usepackage{lmodern}
\usepackage[margin=1in]{geometry}
\usepackage{amsmath,amssymb,amsthm,mathtools}
\usepackage{bm}
\usepackage{array,booktabs}
\usepackage{graphicx}
\usepackage[section]{placeins}
\usepackage{enumitem}
\usepackage{microtype}
\usepackage[numbers,sort&compress]{natbib}
\usepackage{xcolor}
\usepackage{hyperref}
\usepackage[nameinlink,noabbrev]{cleveref}

\hypersetup{
  colorlinks=true,
  linkcolor=blue!55!black,
  citecolor=green!35!black,
  urlcolor=blue!60!black
}

\newtheorem{theorem}{Theorem}[section]
\newtheorem{proposition}[theorem]{Proposition}
\newtheorem{lemma}[theorem]{Lemma}
\newtheorem{corollary}[theorem]{Corollary}
\theoremstyle{definition}

\theoremstyle{remark}

\newcommand{\R}{\mathbb{R}}
\newcommand{\Z}{\mathbb{Z}}
\newcommand{\Sphere}{\mathbb{S}}

\newcommand{\ip}[2]{\left\langle #1,#2\right\rangle}
\newcommand{\norm}[1]{\left\lVert #1\right\rVert}
\newcommand{\abs}[1]{\left\lvert #1\right\rvert}
\newcommand{\proj}[1]{P_{#1}^{\perp}}
\newcommand{\abar}{\underline a_{\beta,n}}
\newcommand{\Dini}{D^{+}}
\newcommand{\diag}{\operatorname{diag}}
\newcommand{\spec}{\operatorname{spec}}

\newcommand{\distS}{d_{\Sphere}}
\newcommand{\one}{\mathbf{1}}
\newcommand{\e}{\mathrm{e}}
\newcommand{\ii}{\mathrm{i}}
\newcolumntype{P}[1]{>{\raggedright\arraybackslash}p{#1}}
\crefname{assumption}{assumption}{assumptions}

\hypersetup{
  pdftitle={Self-Attention Dynamics with Rotary Position Embeddings: Twisted States and Explicit Consensus Rates on the Sphere},
  pdfauthor={Hao Ye},
  pdfsubject={Dynamical systems for spherical self-attention with rotary position embeddings},
  pdfkeywords={self-attention dynamics, rotary position embeddings, consensus on spheres, twisted states, spectral gaps}
}

\title{Self-Attention Dynamics with Rotary Position Embeddings:\\
Twisted States and Explicit Consensus Rates on the Sphere}
\author{Hao Ye\\[0.35em]{\small Xi'an Institute of Optics and Precision Mechanics, Chinese Academy of Sciences}\\{\small University of Chinese Academy of Sciences}\\{\small\texttt{yehao22@mails.ucas.ac.cn}}}
\date{July 2026}

\begin{document}

\maketitle
\begin{abstract}
Rotary position embeddings (RoPE) modify attention scores through
position-dependent rotations, but their effect on normalized token dynamics is
not captured by the vanilla spherical self-attention model.  We study the
continuous-time dynamics obtained when queries and keys are rotated while
values remain on the unit sphere.  The resulting attention kernel is reversible
and admits a sharp uniform softmax floor, yet the natural RoPE interaction
energy has derivatives of both signs within one fixed nontrivial system.  Every
consensus state remains an equilibrium, and its transverse linearization is a
reversible Markov operator whose kernel depends on the consensus point through
its energy across RoPE planes.  On a resonant single-frequency ring we derive an
exact Bessel-aliasing spectrum, including non-coprime frequencies and the
correct fixed-ring large-$\beta$ asymptotics.  Globally, closed hemispheres
are invariant, while pairwise non-obtuse configurations and strict open
semicircles contract with explicit half-angle and single-point tail bounds.
These regional estimates instantiate a kernel-generic positivity principle
with the sharp RoPE softmax floor.  RoPE also selects an explicit
score-flattening twisted branch; the generic resonant family is
non-hyperbolic and linearly unstable, whereas an odd antipodal family becomes a
hyperbolic saddle after quotienting global rotation.  In multiple dimensions,
the local consensus gap can depend non-monotonically on the allocation of
energy across frequency planes, so no universal ordering by frequency is
valid.  Independent matrix, finite-difference, and nonlinear-flow computations
cross-check the theorem boundaries and the reported constants.
\end{abstract}

\section{Introduction}
\label{sec:introduction}

Self-attention is the interaction mechanism at the core of the Transformer
\citep{vaswani2017attention}.  In a deep residual stack, the layer index can be
idealized as time, turning token representations into an interacting particle
system.  This viewpoint has exposed clustering, metastability, and
synchronization phenomena that are difficult to see from a layerwise
algebraic description alone
\citep{geshkovski2023emergence,geshkovski2023mathematical,geshkovski2024dynamic,rigollet2025meanfield}.
Normalization is especially important in this limit because it confines the
particles to a sphere, changes contraction speeds, and can expose gradient or
spectral structure
\citep{burger2025meanfield,karagodin2025normalization,kuehn2026spectral}.

Rotary position embedding encodes position by rotating query and key blocks,
thereby converting absolute rotations into relative phase differences inside
the attention score \citep{su2021roformer}.  This geometric modification is not
an innocuous relabeling of the spherical dynamics: the rotated score controls
the weights, while the unrotated value vectors determine the motion.  Related
analyses show that masking, normalization, and positional encoding can each
alter asymptotic token behavior
\citep{wu2024role,karagodin2024clustering,rodriguezabella2024asymptotic,pham2025dynamical},
but the normalized query/key-only RoPE flow requires its own equilibrium,
spectrum, and convergence analysis.

The central observation is that RoPE preserves a symmetric positive kernel at
consensus even though it precludes an energy-sign law uniform over the RoPE
model class away from consensus.
This separation leads to two complementary conclusions.  Locally, consensus is
governed by the spectral gap of a reversible Markov kernel and can be much
slower than vanilla attention.  Globally, convergence remains provable in
contractive geometric regions without invoking a Lyapunov function.  The same
phase geometry also produces twisted and antipodal equilibria that are absent
from the simplest consensus narrative.

The paper makes the following contributions.

\begin{enumerate}[leftmargin=2em]
  \item Starting from a normalized residual attention step, we formulate
  query/key-only RoPE self-attention on $(\Sphere^{d-1})^n$, prove global
  well-posedness, reversibility, and the sharp score-range attention floor, and
  give exact same-parameter witnesses showing that the natural RoPE interaction
  energy has no uniform monotonic sign.
  \item We derive the full tangent linearization at consensus.  On a resonant
  single-frequency ring, we obtain an exact Bessel-aliasing spectrum that
  handles $\gcd(m,n)>1$, and we separate the fixed-ring exponential
  large-$\beta$ regime from an unproved dense-phase joint limit.
  \item We instantiate a kernel-generic positivity argument with the sharp RoPE
  floor to obtain forward invariance of fixed closed hemispheres and explicit
  convergence rates in pairwise non-obtuse regions and strict open
  semicircles, including a geodesic tail bound to the limiting consensus
  point.
  \item We characterize a RoPE-locked, score-flattening twisted branch and its
  linear spectrum, including the odd antipodal exception, and prove that a
  multi-frequency consensus kernel is positive semidefinite with a potentially
  non-monotone dependence on the consensus-plane energies.
\end{enumerate}

Independent numerical cross-checks complement the proofs for the most
convention-sensitive formulas, including Bessel aliasing, twisted-state
Jacobians and drift, Dini inequalities, positive semidefiniteness, frequency
aliasing, and nonlinear decay rates.

The next section compares these claims at theorem level with normalized
attention dynamics, positional-encoding analyses, and oscillator
interpretations.  Our results do not assert synchronization from arbitrary
initial data or claim that twisted states, Toeplitz score structure, or the
RoPE--oscillator dictionary are new in themselves.  The contribution is the
exact asymptotic theory of the normalized query/key-only flow: its
RoPE-selected equilibrium branch, row-normalized consensus spectrum, and
regional convergence rates.

\section{Related work and contribution boundary}
\label{sec:related-work}

\paragraph{Normalized and spherical attention dynamics.}
Continuous-depth self-attention has been studied as a finite-particle or
mean-field system, with results on clustering, metastability, normalization,
and gradient-flow structure
\citep{geshkovski2023emergence,geshkovski2024dynamic,
burger2025meanfield,karagodin2025normalization}.  Recent finite-particle work
classifies consensus, bipartite, clustered, and polygonal equilibria
\citep{altafini2025multistability}, analyzes spectral selection through the
learned value matrix under symmetric weights \citep{kuehn2026spectral}, and
extends energy arguments to multiple heads \citep{pendharkar2026gradient}.
These results establish that polygonal equilibria and spherical spectral
analysis are not specific to RoPE.  Our identity-value model instead breaks
permutation symmetry through fixed query/key rotations.  Its spectrum is that
of a row-normalized \emph{position} kernel at consensus, and its twisted
branch is selected by the condition that RoPE-adjusted scores become exactly
flat.

\begin{table}[t]
\centering
\caption{Closest-work comparison for standard RoPE
(rotated queries/keys and unrotated values).}
\label{tab:closest-work}
\small
\begin{tabular}{@{}P{0.20\linewidth}P{0.22\linewidth}P{0.24\linewidth}P{0.25\linewidth}@{}}
\toprule
Work & State/model & Main overlap & Boundary relative to this paper \\
\midrule
\citet{karagodin2025normalization}; \citet{kuehn2026spectral}
& Unit-sphere attention without positional encoding
& Normalized flow, rates, equilibria, spectral structure
& No RoPE position kernel or Bessel-aliasing gap \\
\citet{altafini2025multistability}
& Spherical self-attention with learned value matrix
& Polygonal equilibria and instability
& No RoPE-locked flat-score branch or its resonance spectrum \\
\citet{pham2025dynamical}
& Unnormalized Euclidean dynamics with positional encoding
& Continuous-time RoPE token dynamics
& Norm collapse/divergence rather than fixed-norm equilibrium and consensus rates \\
\citet{gu2026positional}; \citet{liu2026phase}
& Static logits and signal-level RoPE theory
& Multiplicative/Toeplitz structure, phase modulation, aliasing
& No normalized state flow or tangent Markov generator \\
\citet{nunley2026kuramoto}
& Modified causal torus architecture
& RoPE phase drift and Kuramoto coupling
& No asymptotic theory for the standard spherical RoPE flow \\
This paper
& Standard query/key-only RoPE on $(\Sphere^{d-1})^n$
& Equilibria, local spectrum, and convergence
& Exact non-coprime Bessel spectrum, regional rates, and multi-plane non-monotonicity \\
\bottomrule
\end{tabular}
\end{table}

Projected consensus and synchronization on spheres also predate attention
\citep{thunberg2018lifting,criscitiello2024synchronization}.  Accordingly, the
regional results in \cref{sec:global-convergence} are not presented as the
first hemispherical synchronization theorem.  Their role here is to give a
kernel-generic contraction principle with explicit constants after inserting
the sharp, state-uniform RoPE softmax floor.

\paragraph{Positional encoding and RoPE.}
\citet{pham2025dynamical} directly study positional encoding in
continuous-time token dynamics, including RoPE, but in an unnormalized
Euclidean model with general learned matrices and norm collapse or divergence
as the main asymptotic alternatives.  Our states have fixed unit norm, values
are unrotated, and the questions are equilibria, tangent rates, and regional
consensus.  Static analyses of multiplicative positional information and RoPE
phase modulation already expose Toeplitz logit structure, frequency
periodicity, and aliasing \citep{gu2026positional,liu2026phase}.  We therefore
do not claim novelty for a Toeplitz or Fourier representation by itself.  The
object analyzed here is the row-normalized consensus Markov kernel and, on a
resonant ring, its exact Bessel-aliasing spectrum, including unreachable modes for
non-coprime frequencies.  The observation that standard RoPE leaves the value
path position-blind, recently emphasized by the RoVE architecture
\citep{garciacastellanos2026rove}, is also the precise modeling boundary used
here.

\paragraph{Oscillator interpretations and twisted states.}
Kuramoto Attention makes the phase-coupling interpretation explicit in a
causal, discrete torus architecture \citep{nunley2026kuramoto}; its frustrated
variant modifies the value law using learned harmonic and delayed couplings
\citep{nunley2026frustrated}.  Thus the dictionary between rotary phases and
oscillator coupling is an existing modeling connection, not a claim of this
paper.  Twisted states likewise have a long history in oscillator networks
\citep{goebel2021twisted}.  What is specific here is an asymptotic result for
the standard query/key-only spherical flow: an if-and-only-if resonance
condition for the score-flattening branch, its exact Jacobian spectrum, and
the separate odd antipodal family.

To our knowledge, the combined theorem package in the final row is not
contained in the works above.  This statement is deliberately narrower than
claiming the first RoPE dynamics, the first oscillator interpretation, or the
first polygonal equilibrium.

\section{RoPE spherical self-attention}
\label{sec:model}

Let $n\ge2$ and let $d=2r\ge2$.  Token $i$ has a fixed position
$p_i\in\R$ and a state $x_i(t)\in\Sphere^{d-1}$.  Given finite frequencies
$\Omega=(\omega_1,\ldots,\omega_r)$, define the block rotation
\begin{equation}
R_p=\bigoplus_{\ell=1}^{r}
\begin{pmatrix}
\cos(\omega_\ell p)&-\sin(\omega_\ell p)\\
\sin(\omega_\ell p)& \cos(\omega_\ell p)
\end{pmatrix}.
\label{eq:rope-rotation}
\end{equation}
For finite $\beta\ge0$, the score, symmetric kernel, degree, and attention
weights are
\begin{align}
s_{ij}(x)&=\ip{R_{p_i}x_i}{R_{p_j}x_j},
&W_{ij}(x)&=\e^{\beta s_{ij}(x)},\nonumber\\
Z_i(x)&=\sum_{k=1}^{n}W_{ik}(x),
&A_{ij}(x)&=\frac{W_{ij}(x)}{Z_i(x)}.
\label{eq:attention-def}
\end{align}
RoPE acts only on queries and keys in \eqref{eq:attention-def}; the value
vectors remain unrotated.  This is the standard RoPE value pathway rather than
the rotated-value modification of \citet{garciacastellanos2026rove}.

The continuous model is the first-order limit of a normalized residual layer.
Let
\begin{equation}
y_i^\ell=\sum_{j=1}^{n}A_{ij}(x^\ell)x_j^\ell,
\qquad
x_i^{\ell+1}
=\frac{x_i^\ell+h y_i^\ell}
       {\norm{x_i^\ell+h y_i^\ell}},
\label{eq:normalized-residual-step}
\end{equation}
where $h>0$ is the residual step.  Since $\norm{x_i^\ell}=1$,
\begin{equation}
\frac{x_i^{\ell+1}-x_i^\ell}{h}
=
y_i^\ell-\ip{x_i^\ell}{y_i^\ell}x_i^\ell+O(h)
=\proj{x_i^\ell}y_i^\ell+O(h).
\label{eq:normalized-residual-limit}
\end{equation}
The remainder is uniform on the compact product sphere.  Thus, under the
controlled specialization $Q=K=V=I$ apart from the prescribed query/key RoPE
rotations, the depth-time limit of \eqref{eq:normalized-residual-step} is
\begin{equation}
\dot x_i
=\proj{x_i}\sum_{j=1}^{n}A_{ij}(x)x_j
=\sum_{j=1}^{n}A_{ij}(x)
\bigl(x_j-\ip{x_i}{x_j}x_i\bigr),
\qquad \proj{x}=I-xx^\top.
\label{eq:sphere-flow}
\end{equation}
after absorbing the fixed score scaling into $\beta$.  Related normalized
residual limits are used in spherical attention dynamics
\citep{karagodin2025normalization,kuehn2026spectral}; here
\eqref{eq:normalized-residual-limit} fixes the precise RoPE and value-path
conventions used below.

\begin{theorem}[Well-posedness, reversibility, and sharp floor]
\label{thm:wellposed-balance-floor}
For every initial state in $(\Sphere^{d-1})^n$, \eqref{eq:sphere-flow} has a
unique global solution.  At every state, $W=W^\top$, $A\one=\one$, and
\begin{equation}
\pi_iA_{ij}=\pi_jA_{ji},
\qquad
\pi_i=\frac{Z_i}{\sum_kZ_k}.
\label{eq:detailed-balance}
\end{equation}
Moreover,
\begin{equation}
A_{ij}(x)\ge \abar
:=\frac{1}{1+(n-1)\e^{2\beta}}
=\frac{\e^{-2\beta}}{n-1+\e^{-2\beta}}.
\label{eq:sharp-floor}
\end{equation}
The constant is sharp over the model class.  At $\beta=0$ it equals $1/n$;
for $\beta>0$ it is strictly stronger than the coarser
$\e^{-2\beta}/n$ bound.
\end{theorem}

\Cref{thm:wellposed-balance-floor} is proved in
\cref{app:structural-proofs}.  Compactness of the product sphere supplies
global existence, while \eqref{eq:sharp-floor} is the optimal softmax bound for
scores in $[-1,1]$.  Detailed balance does not make $A$ symmetric unless the
degrees $Z_i$ are equal.

\subsection{Failure of the natural RoPE energy law}

On $\Sphere^1$, write $x_i=(\cos\theta_i,\sin\theta_i)$ and consider one
frequency $\omega$.  With
\begin{equation}
\delta_{ij}=\theta_i-\theta_j+\omega(p_i-p_j),
\qquad
E_+(\theta)=\sum_{i,j=1}^{n}\e^{\beta\cos\delta_{ij}},
\label{eq:natural-energy}
\end{equation}
the phase equation induced by \eqref{eq:sphere-flow} is
\begin{equation}
\dot\theta_i=\sum_jA_{ij}(\theta)\sin(\theta_j-\theta_i).
\label{eq:phase-flow}
\end{equation}

The energy in \eqref{eq:natural-energy} is ``natural'' because it recovers the
standard positive-sign interaction law when the rotary phase vanishes.  Indeed,
with
\[
u_i=\sum_jW_{ij}\sin\delta_{ij},
\qquad
v_i=\sum_jW_{ij}\sin(\theta_i-\theta_j),
\]
the exact identity
\begin{equation}
\dot E_+=2\beta\sum_i\frac{u_iv_i}{Z_i}
\label{eq:energy-derivative-summary}
\end{equation}
holds.  At $\omega=0$, $u_i=v_i$ and hence
\begin{equation}
\dot E_+=2\beta\sum_i\frac{u_i^2}{Z_i}\ge0.
\label{eq:vanilla-energy-sign}
\end{equation}
RoPE separates $u_i$ from $v_i$, so the argument no longer supplies a sign.
The following proposition shows that this loss is genuine rather than merely
a gap in that calculation.

\begin{proposition}[Same-system energy no-go]
\label{prop:energy-no-go}
Fix $n=3$, $(p_1,p_2,p_3)=(0,1,2)$, $\beta=1$, and $\omega=\pi$.
Then
\begin{equation}
\left.\frac{\mathrm d}{\mathrm dt}E_+\right|_{\theta=(0,\pi/6,\pi/3)}>0,
\qquad
\left.\frac{\mathrm d}{\mathrm dt}E_+\right|_{\theta=(0,\pi/6,0)}<0.
\label{eq:energy-sign-witnesses}
\end{equation}
Consequently, for this fixed nontrivial RoPE system, neither $E_+$ nor $-E_+$
is a Lyapunov function.  In particular, this energy form has no monotonicity law
of either sign that is uniform over the whole RoPE model class.
\end{proposition}

This proposition is energy-specific: it neither excludes a different
Lyapunov function nor claims sign-indefiniteness for every frequency.  The
exact closed forms behind \eqref{eq:energy-sign-witnesses} appear in
\cref{app:structural-proofs}.

\section{Consensus, twisted equilibria, and local linearization}
\label{sec:equilibria}

\subsection{Consensus and its position kernel}

Let $x^\star\in\Sphere^{d-1}$ and decompose it into its $r$ rotation planes,
$x^\star=(x^\star_{(1)},\ldots,x^\star_{(r)})$.  Define
\begin{equation}
a_\ell=\norm{x^\star_{(\ell)}}^2,
\qquad a_\ell\ge0,
\qquad \sum_{\ell=1}^{r}a_\ell=1.
\label{eq:plane-energies}
\end{equation}

\begin{proposition}[Consensus kernel]
\label{prop:consensus-kernel}
The consensus manifold
\[
\mathcal C=\{(x^\star,\ldots,x^\star):x^\star\in\Sphere^{d-1}\}
\]
consists of equilibria.  At such a point,
\begin{equation}
W^\star_{ij}=\exp\!\bigl(\beta K_a(p_i-p_j)\bigr),
\qquad
K_a(h)=\sum_{\ell=1}^{r}a_\ell\cos(\omega_\ell h).
\label{eq:consensus-kernel}
\end{equation}
Thus the consensus kernel depends on $x^\star$ only through the plane-energy
vector $a$, not through the within-plane phases.
\end{proposition}

\subsection{An explicit twisted family on the circle}

Assume $d=2$ and consecutive integer positions $p_j=j$ for
$j=0,\ldots,n-1$.  The twisted configuration is
\begin{equation}
\theta_j=c-\omega j\pmod{2\pi}.
\label{eq:twisted-state}
\end{equation}
Its RoPE-adjusted scores are all equal, so $A_{ij}=1/n$.

\begin{proposition}[Exact twisted equilibria]
\label{prop:twisted-equilibria}
The configuration \eqref{eq:twisted-state} is an equilibrium if and only if
\begin{equation}
\omega\in\pi\Z
\qquad\text{or}\qquad
n\omega\in2\pi\Z.
\label{eq:twisted-iff}
\end{equation}
Odd multiples of $\pi$ give a bipodal configuration; even multiples give
consensus.  Outside \eqref{eq:twisted-iff}, for every fixed
$\omega\notin2\pi\Z$,
\begin{equation}
\max_i\abs{\dot\theta_i}
\le
\frac{\abs{\sin(n\omega/2)}}{n\abs{\sin(\omega/2)}}
\le\frac{1}{n\abs{\sin(\omega/2)}}
=O_\omega(n^{-1}).
\label{eq:twisted-drift}
\end{equation}
The constant is not uniform when the frequency varies with $n$ toward
$2\pi\Z$.
\end{proposition}

\subsection{Consensus linearization}

\begin{theorem}[Tangent linearization]
\label{thm:consensus-linearization}
Let $x_i=x^\star$ be a consensus equilibrium and let
$\xi_i\in T_{x^\star}\Sphere^{d-1}$.  The tangent linearization is
\begin{equation}
\dot\xi_i=\sum_{j=1}^{n}A^\star_{ij}\xi_j-\xi_i.
\label{eq:tangent-linearization}
\end{equation}
In a common orthonormal tangent basis,
\begin{equation}
J_{\mathcal C}=(A^\star-I_n)\otimes I_{d-1}.
\label{eq:kronecker-jacobian}
\end{equation}
For the fixed linear system, the stationary-weighted tangent mean
\begin{equation}
\bar\xi_\pi=\sum_i\pi_i^\star\xi_i,
\qquad
\pi_i^\star=\frac{Z_i^\star}{\sum_kZ_k^\star},
\label{eq:linear-weighted-mean}
\end{equation}
is conserved.
\end{theorem}

The conservation in \eqref{eq:linear-weighted-mean} is only a linearized
identity.  The nonlinear quantity $\sum_i Z_i(x)x_i$ is not asserted to be
conserved.

\section{Resonant-ring spectra and consensus slowdown}
\label{sec:resonant-spectrum}

Consider $d=2$, $p_j=j$, and the exact resonance with $m\in\Z$,
\begin{equation}
\omega=\frac{2\pi m}{n},
\qquad
g=\gcd(m,n),
\qquad
L=\frac{n}{g}.
\label{eq:resonance-data}
\end{equation}
Here $L$ is the number of distinct sampled phases.  Exact resonance makes the
consensus kernel circulant; without it, a finite chain must not be replaced by
a periodic distance.

\begin{theorem}[Bessel-aliasing spectrum]
\label{thm:bessel-spectrum}
Under \eqref{eq:resonance-data}, $A^\star$ is symmetric, doubly stochastic, and
circulant.  Its Fourier eigenvalues are
\begin{equation}
\lambda_q=
\frac{
\displaystyle\sum_{\substack{k\in\Z\\mk\equiv q\ (\mathrm{mod}\ n)}}
I_k(\beta)}{
\displaystyle\sum_{\substack{k\in\Z\\mk\equiv0\ (\mathrm{mod}\ n)}}
I_k(\beta)},
\qquad q=0,\ldots,n-1,
\label{eq:bessel-alias-spectrum}
\end{equation}
where $I_k$ is the modified Bessel function of the first kind.  If $g\nmid q$,
then $\lambda_q=0$.  For $\beta>0$, exactly $L$ eigenvalues are positive and
$n-L$ vanish.  At $\beta=0$, the spectrum is $(1,0,\ldots,0)$.
\end{theorem}

The tangent consensus rate is
\begin{equation}
\gamma=1-\lambda_2(A^\star),
\label{eq:gap-definition}
\end{equation}
where eigenvalues are ordered non-increasingly.  Formula
\eqref{eq:bessel-alias-spectrum} handles non-coprime $m$; dropping the
congruence reachability condition gives the wrong zero-eigenvalue
multiplicity.

\begin{theorem}[Positive spectrum and strict slowdown]
\label{thm:psd-slowdown}
For the general multi-frequency consensus kernel
\eqref{eq:consensus-kernel},
\begin{equation}
\spec(A^\star)\subset[0,1],
\qquad 0<\gamma\le1.
\label{eq:positive-spectrum}
\end{equation}
If $\beta=0$, then $\gamma=1$.  If $\beta>0$, equality $\gamma=1$ holds if
and only if every active plane is invisible on the sampled positions:
\begin{equation}
a_\ell>0
\quad\Longrightarrow\quad
\omega_\ell(p_i-p_j)\in2\pi\Z
\quad\text{for all }i,j.
\label{eq:invisible-planes}
\end{equation}
Whenever at least one active plane is visible, $0<\gamma<1$.
\end{theorem}

The proof writes $K_a$ as a Gram matrix and applies the Schur product theorem
to its entrywise exponential.  The row-normalized matrix is similar to
$D^{-1/2}W^\star D^{-1/2}\succeq0$.  Condition
\eqref{eq:invisible-planes}, rather than ``nonzero frequency,'' is the exact
strictness criterion.

\begin{theorem}[Fixed-effective-period large-$\beta$ asymptotics]
\label{thm:finite-ring-asymptotics}
Fix the effective period $L$ in \eqref{eq:resonance-data}.
\begin{enumerate}[label=(\roman*),leftmargin=2em]
  \item If $L=1$, the kernel is constant and $\gamma(\beta)=1$.
  \item If $L=2$, then
  \begin{equation}
  \gamma(\beta)=1-\tanh\beta
  =\frac{2\e^{-2\beta}}{1+\e^{-2\beta}}
  \sim2\e^{-2\beta}.
  \label{eq:antialigned-gap}
  \end{equation}
  \item If $L\ge3$, let $\Delta_L=1-\cos(2\pi/L)$.  As
  $\beta\to\infty$,
  \begin{equation}
  \gamma(\beta)\sim2\Delta_L\e^{-\beta\Delta_L}.
  \label{eq:fixed-ring-asymptotic}
  \end{equation}
\end{enumerate}
\end{theorem}

Figure~\ref{fig:resonant-spectrum-slowdown} separates three facts that are
easy to conflate in a gap-only calculation: the exact Fourier-mode
reachability pattern, the absolute fixed-ring slowdown, and the leading
asymptotic prefactor.  All plotted finite-ring gaps use the nonnegative-sum
Fourier formula, so the smallest values are not formed by subtracting two
nearby eigenvalues.

\begin{figure}[t]
\centering
\includegraphics[width=\textwidth]{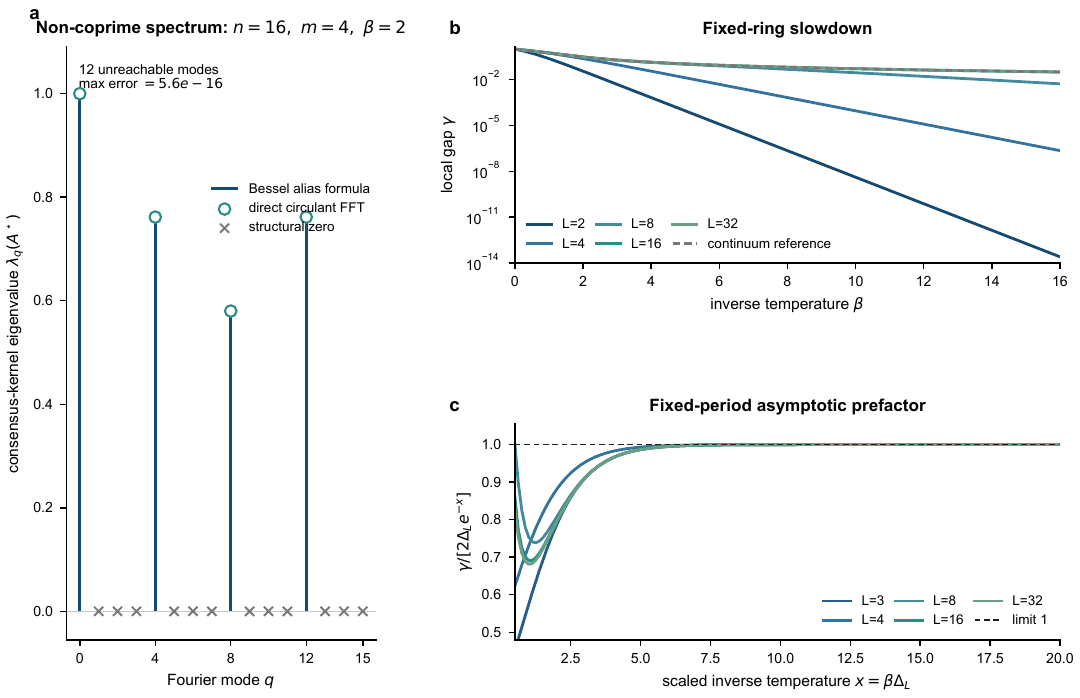}
\caption{\textbf{Exact resonant spectrum and fixed-ring slowdown.}
\textbf{a}, Complete Fourier-labelled spectrum for
$n=16$, $m=4$, and $\beta=2$.  Blue stems are the
congruence-filtered Bessel formula, open teal circles are the direct FFT of
the exact circulant attention row, and grey crosses mark the $12$ modes
excluded by the exact $\gcd(m,n)$ reachability mask.
\textbf{b}, Cancellation-free finite-ring gaps for the displayed effective
periods $L$; the dashed continuum expression $1-I_1(\beta)/I_0(\beta)$ is a
reference only and is not a proved joint-limit approximation here.
\textbf{c}, Fixed-$L$ ratios
$\gamma/[2\Delta_L\exp(-\beta\Delta_L)]$ against
$\beta\Delta_L$ for $L\ge3$, with the unit asymptotic limit dashed.  Each
curve is an exact finite-ring computation and does not assert a bound uniform
in $L$.  All quantities are deterministic, so no error bars are applicable.}
\label{fig:resonant-spectrum-slowdown}
\end{figure}

The continuum expression $1-I_1(\beta)/I_0(\beta)\sim1/(2\beta)$ is not the
fixed-$L$ asymptotic.  Recovering it in a joint limit
$\beta,L\to\infty$ with $L/\sqrt\beta\to\infty$ requires a uniform aliasing
error bound.  That estimate is not established here, so the joint-limit claim
is not stated as a theorem.

\section{Global contraction in invariant geometric regions}
\label{sec:global-convergence}

Throughout this section, positions and frequencies are arbitrary and fixed,
and $\abar$ denotes \eqref{eq:sharp-floor}.  The conclusions are independent
of resonance.  The proofs use only row stochasticity and the uniform lower
bound $A_{ij}\ge\abar$; they are therefore a kernel-generic contraction
principle instantiated with the sharp RoPE floor.  Earlier sphere-consensus
theory gives broader qualitative synchronization results under other graph and
interaction assumptions
\citep{thunberg2018lifting,criscitiello2024synchronization}.  The narrower
regions below supply explicit state-uniform constants and tail bounds for the
present attention kernel.

\begin{proposition}[Fixed closed hemispheres are invariant]
\label{prop:hemisphere-invariance}
Fix $w\in\Sphere^{d-1}$ and define
\begin{equation}
h_i(t)=\ip{x_i(t)}{w},
\qquad m_w(t)=\min_i h_i(t).
\label{eq:hemisphere-margin}
\end{equation}
If $m_w(0)\ge0$, then $m_w(t)\ge m_w(0)\ge0$ for all $t\ge0$.
Thus both the fixed closed hemisphere and its interior are forward invariant.
This conclusion alone does not imply consensus.
\end{proposition}

\begin{theorem}[Pairwise non-obtuse contraction]
\label{thm:acute-contraction}
Let
\begin{equation}
c(t)=\min_{i<j}\ip{x_i(t)}{x_j(t)}.
\label{eq:minimum-inner-product}
\end{equation}
If $c(0)=c_0\ge0$, then $c(t)\ge0$ and
\begin{equation}
\Dini c(t)\ge2\abar\bigl(1-c(t)^2\bigr).
\label{eq:acute-dini}
\end{equation}
With
\begin{equation}
\rho(t)=\frac{1-c(t)}{1+c(t)},
\qquad
\rho_0=\frac{1-c_0}{1+c_0},
\label{eq:rho-def}
\end{equation}
one has
\begin{equation}
\rho(t)\le\rho_0\e^{-4\abar t}.
\label{eq:rho-decay}
\end{equation}
Equivalently, for the pairwise geodesic diameter
$D_{\Sphere}(t)=\arccos c(t)$,
\begin{equation}
\tan\frac{D_{\Sphere}(t)}2
\le
\tan\frac{D_{\Sphere}(0)}2\,\e^{-2\abar t}.
\label{eq:acute-half-angle}
\end{equation}
If $c_0=0$, then $c(t)\ge\tanh(2\abar t)>0$ for every $t>0$.
\end{theorem}

\begin{theorem}[Strict open-semicircle contraction]
\label{thm:semicircle-contraction}
Assume $d=2$ and choose continuous lifts $\theta_i(t)\in\R$ of the phases.
If
\begin{equation}
D_0=\max_i\theta_i(0)-\min_i\theta_i(0)<\pi,
\label{eq:strict-semicircle}
\end{equation}
then the lifted diameter remains below $\pi$ and satisfies
\begin{equation}
\Dini D(t)\le-2\abar\sin D(t),
\qquad
\tan\frac{D(t)}2\le\tan\frac{D_0}2\,\e^{-2\abar t}.
\label{eq:semicircle-bound}
\end{equation}
\end{theorem}

\begin{proposition}[Bipodal boundary counterexample]
\label{prop:bipodal-boundary}
Let $u\in\Sphere^{d-1}$ and choose signs $\sigma_i\in\{-1,1\}$ with both
signs present.  Then $x_i=\sigma_i u$ is a stationary non-consensus
configuration for every finite $\beta$, every position set, and every frequency
spectrum.  In $d=2$ this has lifted diameter $\pi$, so the strict condition in
\cref{thm:semicircle-contraction} cannot be weakened to $D_0\le\pi$.
\end{proposition}

The normalized comparisons in
\cref{fig:regional-contraction} use the observable appearing in the proofs
rather than fitting an unconstrained exponential.  They show both the
conservatism of the state-uniform floor and the sharp failure at the antipodal
boundary.

\begin{figure}[t]
\centering
\includegraphics[width=\textwidth]{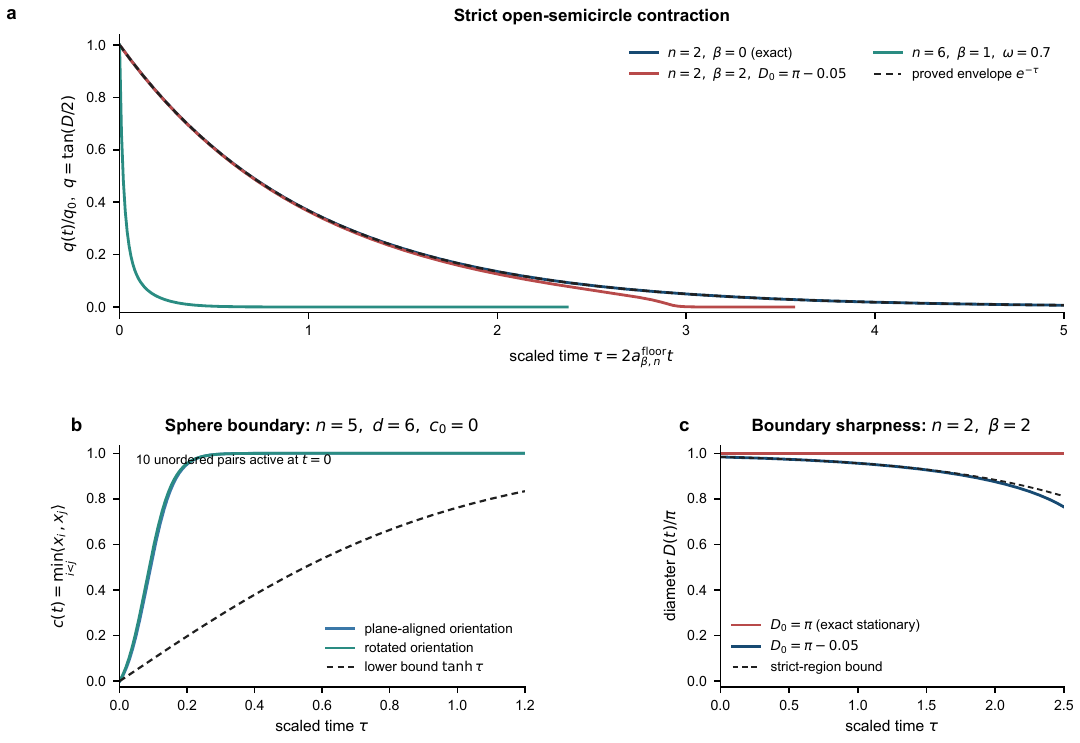}
\caption{\textbf{Regional contraction and its sharp boundary.}
\textbf{a}, Deterministic circle trajectories plotted as
$q(t)/q_0$, where $q=\tan(D/2)$, against
$\tau=2\abar t$.  The dashed curve is the proved envelope
$\exp(-\tau)$; the $\beta=0$, $n=2$ case is exact.  Numerical curves are
stopped after $q(t)/q_0<10^{-10}$ to avoid displaying the floating-point
noise floor.
\textbf{b}, Minimum pairwise inner product for two fixed orientations of
five initially orthogonal tokens in $d=6$, with $\beta=1$ and the standard
three-plane RoPE frequencies.  All ten unordered pairs are active at
$c_0=0$; both step-halved trajectories lie above the dashed lower bound
$\tanh\tau$.
\textbf{c}, The exact bipodal state with $D_0=\pi$ remains stationary,
whereas the strict case $D_0=\pi-0.05$ contracts below its comparison
solution for $n=2$, $\beta=2$, and $\omega=0$.  These are deterministic ODE
computations, not statistical samples; no error bars are applicable.}
\label{fig:regional-contraction}
\end{figure}

\begin{corollary}[Single-point limit and explicit tail]
\label{cor:single-point-tail}
Under the assumptions of either \cref{thm:acute-contraction} or
\cref{thm:semicircle-contraction}, all tokens converge to a common point
$x_\infty\in\Sphere^{d-1}$.  Let $D(t)$ be the corresponding maximal geodesic
diameter and set
\begin{equation*}
q_0=\tan\frac{D(0)}2,
\qquad Q(t)=q_0\e^{-2\abar t}.
\end{equation*}
Then
\begin{equation}
\distS(x_i(t),x_\infty)
\le\frac{1}{\abar}\operatorname{arsinh}Q(t)
\le\frac{q_0}{\abar}\e^{-2\abar t}.
\label{eq:geodesic-tail}
\end{equation}
\end{corollary}

The proof in \cref{app:global-proofs} treats the nonsmooth minima with finite
active-set Dini calculus.  In particular, negative-part Gr\"onwall estimates
close the boundary cases $m_w(0)=0$ and $c_0=0$ rather than relying on a
formal tangency statement alone.

\paragraph{Rate bookkeeping.}
Let $Y=1-c$ and let $L=\sqrt{2(1-c)}$ be the associated chordal diameter.
Then $Y=L^2/2$ identically.  The half-angle bounds imply
\begin{equation}
D(t)\le2q_0\e^{-2\abar t},
\qquad
L(t)\le2q_0\e^{-2\abar t},
\qquad
Y(t)\le2q_0^2\e^{-4\abar t}.
\label{eq:observable-rates}
\end{equation}
Thus the logarithmic rate of the quadratic deficit $1-c$ is twice that of the
chordal or angular diameter.  This factor is kept explicit in all numerical
rate comparisons.

\paragraph{Local versus uniform rates.}
The consensus gap $\gamma$ in \eqref{eq:gap-definition} is the exact slowest
infinitesimal transverse exponent at a specified consensus point.  By
contrast, $2\abar$ is a state- and position-uniform nonlinear diameter
exponent.  The latter is necessarily conservative.  Indeed, for $\beta>0$,
\[
A^\star
=\abar\one\one^\top+(1-n\abar)B,
\qquad
B=\frac{A^\star-\abar\one\one^\top}{1-n\abar},
\]
where $B$ is row stochastic.  On the quotient by $\operatorname{span}\{\one\}$,
the rank-one term vanishes, so every non-Perron eigenvalue of $A^\star$ has
modulus at most $1-n\abar$.  Together with the positive spectrum from
\cref{thm:psd-slowdown}, this gives
\begin{equation}
\gamma\ge n\abar\ge2\abar.
\label{eq:local-uniform-rate-comparison}
\end{equation}
At $\beta=0$ the same inequality holds with equality
$\gamma=n\abar=1$.  On a fixed resonant ring, the asymptotics make the
separation quantitative:
\begin{align}
\frac{\gamma}{2\abar}
&\sim n-1,
&&L=2,\nonumber\\
\frac{\gamma}{2\abar}
&\sim (n-1)\Delta_L
       \e^{\beta(2-\Delta_L)},
&&L\ge3.
\label{eq:fixed-ring-local-global-ratio}
\end{align}
Thus the exact local rate and the uniform global guarantee answer different
questions and should not be numerically identified.

\section{Twisted-state stability and multi-frequency geometry}
\label{sec:twisted-multifrequency}

\subsection{Linear spectrum of twisted equilibria}

\begin{proposition}[Nontrivial resonant twisted states]
\label{prop:twisted-instability}
Let $p_j=j$, let $m\in\Z$, let $\omega=2\pi m/n$, and let $\bar m$ be
$m$ modulo $n$.
Assume $\bar m\ne0$.  At the twisted equilibrium
$\theta_j=c-\omega j$, the Jacobian is
\begin{equation}
J=\frac1n C,
\qquad
C_{ij}=\cos\!\left(\frac{2\pi m(i-j)}n\right).
\label{eq:twisted-jacobian}
\end{equation}
If $2\bar m\not\equiv0\pmod n$, then
\begin{equation}
\spec(J)=\left\{\tfrac12\ \text{(mult. 2)},\;
0\ \text{(mult. $n-2$)}\right\}.
\label{eq:generic-twisted-spectrum}
\end{equation}
If $n$ is even and $\bar m=n/2$, then
\begin{equation}
\spec(J)=\left\{1\ \text{(mult. 1)},\;
0\ \text{(mult. $n-1$)}\right\}.
\label{eq:even-antipodal-spectrum}
\end{equation}
These equilibria have growing linear modes and center modes, but no linearly
stable modes.  They are therefore non-hyperbolic and linearly unstable, not
saddles in the usual hyperbolic sense.
\end{proposition}

The trivial mode $\bar m=0$ is ordinary consensus and has spectrum
$\{0,-1\text{ (mult. $n-1$)}\}$; it is excluded from
\cref{prop:twisted-instability}.

\begin{proposition}[Odd antipodal exception]
\label{prop:odd-antipodal-saddle}
Let $d=2$, $p_j=j$ for $j=0,\ldots,n-1$, $n=2r+1$, and
$\omega\in(2\Z+1)\pi$.  The alternating twisted circle is an equilibrium
circle, and its Jacobian has spectrum
\begin{equation}
\spec(J)=
\left\{
-\tfrac1n\ \text{(mult. $r$)},\;
0\ \text{(mult. 1)},\;
\tfrac1n\ \text{(mult. $r-1$)},\;
1\ \text{(mult. 1)}
\right\}.
\label{eq:odd-antipodal-spectrum}
\end{equation}
The unique zero mode is global phase rotation.  After quotienting this
symmetry, the equilibrium is a hyperbolic saddle with $r$ stable and $r$
unstable dimensions.
\end{proposition}

The two blocks of \cref{fig:energy-twisted-branch} visualize independent
consequences of the same rotary-score/unrotated-value structure.  The
energy-sign failure does not cause the twisted branch, and the figure uses no
causal implication between them.

\begin{figure}[t]
\centering
\includegraphics[width=0.97\textwidth]{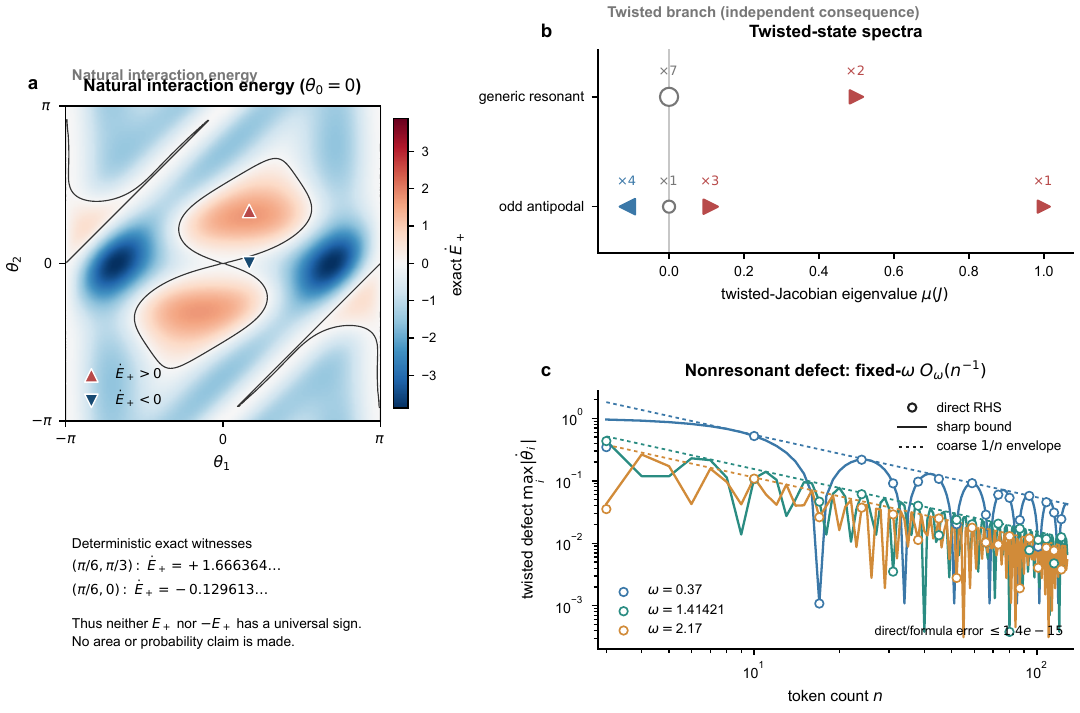}
\caption{\textbf{Natural-energy sign failure and twisted branches.}
\textbf{a}, Exact field of $\dot E_+$ for
$n=3$, $p=(0,1,2)$, $\beta=1$, and $\omega=\pi$, after fixing the global
phase by $\theta_0=0$.  Black curves are the zero level; the two triangular
markers are the deterministic witnesses
$(\theta_1,\theta_2)=(\pi/6,\pi/3)$ and $(\pi/6,0)$.
The field rules out a universal Lyapunov sign for $E_+$ and $-E_+$ only; it
does not rule out other Lyapunov functions and carries no area or probability
interpretation.
\textbf{b}, Exact twisted-Jacobian spectra for the generic resonant case
$(n,m)=(9,2)$ and the odd antipodal case $(n,\omega)=(9,\pi)$.
Blue, grey, and red markers denote stable, center, and growing modes, with
multiplicities shown.  The generic branch is non-hyperbolic and linearly
unstable, whereas the odd antipodal branch is a saddle after quotienting
global rotation.
\textbf{c}, Direct nonresonant drift (open circles), its sharp
geometric-sum bound (solid), and the fixed-frequency
$1/[n|\sin(\omega/2)|]$ envelope (dotted).  The bound is
$O_\omega(n^{-1})$, not uniform as $\omega$ approaches
$2\pi\mathbb Z$.  All quantities are deterministic; no error bars are
applicable.}
\label{fig:energy-twisted-branch}
\end{figure}

\subsection{Multi-frequency consensus spectra}

\begin{theorem}[Multi-frequency kernel and local rate]
\label{thm:multifrequency-spectrum}
Let $x^\star$ be a consensus point with plane energies $a$ as in
\eqref{eq:plane-energies}, and let $W$, $D=\diag(W\one)$, and
$A=D^{-1}W$ be defined by \eqref{eq:consensus-kernel}.  Then $W\succeq0$ and
$A$ is similar to
\begin{equation*}
S=D^{-1/2}WD^{-1/2}\succeq0.
\end{equation*}
Consequently,
\begin{equation}
1=\lambda_1(A)>\lambda_2(A)\ge\cdots\ge\lambda_n(A)\ge0,
\qquad
\gamma(a)=1-\lambda_2(A)\in(0,1].
\label{eq:multifrequency-gap}
\end{equation}
The tangent Jacobian is $(A-I_n)\otimes I_{d-1}$.  If $\beta>0$, equality
$\gamma(a)=1$ holds exactly under the active-plane invisibility condition
\eqref{eq:invisible-planes}.
\end{theorem}

For consecutive integer positions, replacing any frequency $\omega_\ell$ by
$\pm\omega_\ell+2\pi k$ leaves the kernel unchanged.  Thus finite sampling
alone prevents any universal ordering by the numerical magnitude of a
frequency.

\begin{proposition}[Analytic non-monotonicity in plane energy]
\label{prop:analytic-energy-nonmonotonicity}
Take
\begin{equation*}
n=3,\qquad p=(0,1,2),\qquad
\beta=\frac34\log2,\qquad
(\omega_1,\omega_2)=\left(\frac\pi2,\pi\right),
\end{equation*}
and let $a\in[0,1]$ be the energy in the first plane.  Then
\begin{align}
\gamma(0)&=\frac{35\sqrt2-32}{31},\nonumber\\
\gamma(2/3)&=\frac34,\nonumber\\
\gamma(1)&=\frac{u(1+2u)}{1+u+u^2},
\qquad u=2^{-3/4}.
\label{eq:analytic-nonmonotone-values}
\end{align}
In particular,
\begin{equation*}
\gamma(2/3)>\max\{\gamma(0),\gamma(1)\},
\end{equation*}
so $a\mapsto\gamma(a)$ is not generally monotone.
\end{proposition}

The calculation is elementary.  After scaling the diagonal kernel entries to
one, let $x$ and $y$ be the lag-one and lag-two weights.  The two non-Perron
eigenvalues are
\begin{equation}
\lambda_- =\frac{1-y}{1+x+y},
\qquad
\lambda_+ =\frac{1+y}{1+x+y}+\frac{1}{1+2x}-1.
\label{eq:three-token-eigenvalues}
\end{equation}
At $a=2/3$, $x=y=1/2$, so both equal $1/4$.

\begin{proposition}[Analytic frequency-order reversal]
\label{prop:analytic-frequency-reversal}
For $\beta>0$, two tokens at positions $(0,\Delta)$, and one active
frequency $\omega$,
\begin{equation}
\gamma_\omega(\Delta)=\frac{2q}{1+q},
\qquad
q=\exp\!\bigl(\beta[\cos(\omega\Delta)-1]\bigr).
\label{eq:two-token-gap}
\end{equation}
Take $\omega_{\rm low}=\pi/2$ and $\omega_{\rm high}=\pi$.  At $\Delta=1$,
the high frequency is slower; at $\Delta=2$, it is completely aliased and has
$\gamma_{\rm high}=1$, while the low-frequency gap is strictly smaller than
one.  Hence no position-independent ordering by frequency magnitude exists.
\end{proposition}

For the standard two-plane frequencies $(1,0.01)$ at $\beta=4$, direct
finite-context spectra exhibit the same reversal:
\begin{align}
n=256:&\quad
\gamma_{\rm low}=0.240298,\qquad
\gamma_{\rm high}=0.134730,\nonumber\\
n=384:&\quad
\gamma_{\rm low}=0.112756,\qquad
\gamma_{\rm high}=0.135707.
\label{eq:finite-context-reversal}
\end{align}
At $n=384$, the equal-energy mixture has gap $0.263105$, larger than both
endpoints.  These are deterministic observations on the displayed finite
contexts, not a scaling theorem or a prerequisite for the analytic
counterexamples above.

\begin{figure}[t]
\centering
\includegraphics[width=\textwidth]{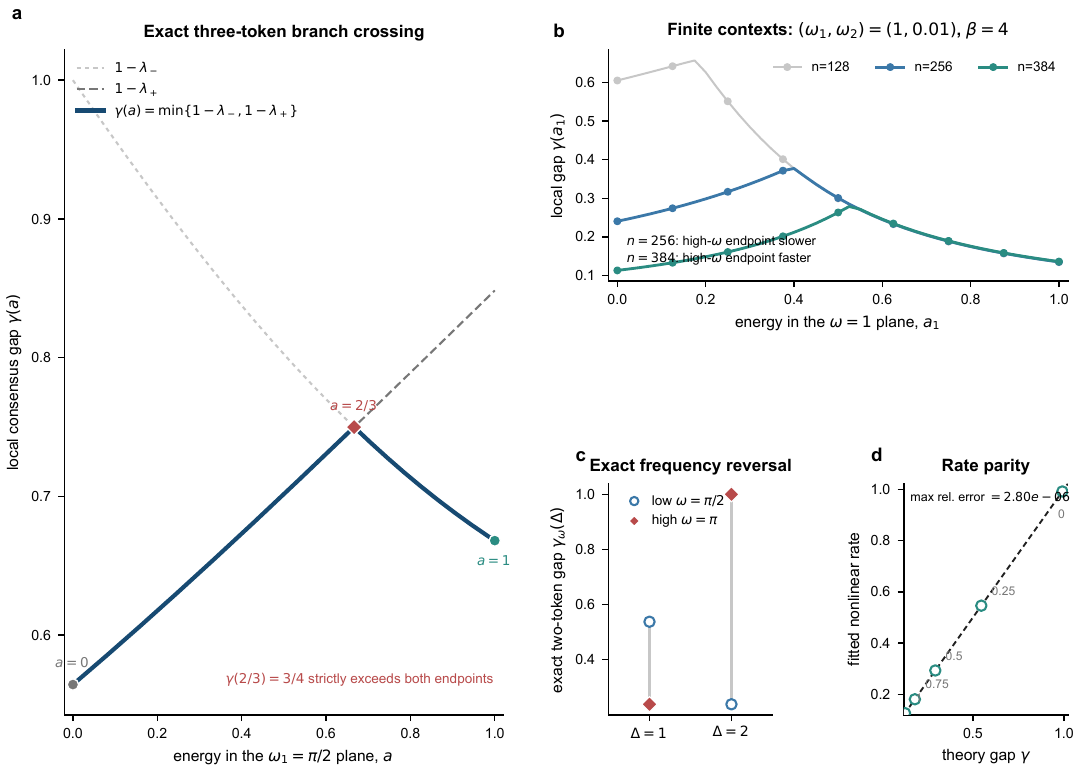}
\caption{\textbf{No universal monotone ordering of multi-plane consensus
rates.}
\textbf{a}, For the exact three-token construction in
\cref{prop:analytic-energy-nonmonotonicity}, the two thin curves are the
branch-limited rates $1-\lambda_-$ and $1-\lambda_+$, and the thick curve is
their minimum $\gamma(a)$.  The exact interior value
$\gamma(2/3)=3/4$ strictly exceeds both endpoints; no claim of a global
maximum is needed.
\textbf{b}, Deterministic contextual observations for
$(\omega_1,\omega_2)=(1,0.01)$ and $\beta=4$.  The plane-energy curves for
$n=256$ and $n=384$ have opposite endpoint orderings, while the $n=128$
curve is shown as a muted reference.  This panel is an observation on the
displayed finite contexts, not a scaling theorem.
\textbf{c}, The exact two-token formula in
\cref{prop:analytic-frequency-reversal} gives opposite low/high-frequency
orderings at position separations $\Delta=1$ and $\Delta=2$.
\textbf{d}, Five deterministic slow-mode integrations at $n=16$ and
$\beta=4$ compare the predicted tangent gap with the fitted nonlinear decay
rate; labels give the high-frequency plane energy.  No error bars are
applicable.}
\label{fig:multifrequency-nonmonotonicity}
\end{figure}

\section{Numerical validation}
\label{sec:numerical-validation}

The computations below cross-check convention-sensitive formulas and quantify
finite-precision effects; no theorem uses a numerical estimate as a proof
step.  Finite-chain position differences are always literal, and periodic
distances are introduced only under exact resonance.  Scaled Bessel functions
and globally scaled kernels avoid overflow without changing normalized
eigenvalues.  Appendix~\ref{app:numerical-protocol} gives the parameter grids,
error definitions, integration tolerances, and rate-fitting windows.

\begin{table}[ht]
\centering
\caption{Independent numerical cross-checks.  Coverage counts parameter
configurations, not statistical replicates.  The last column reports the
largest discrepancy unless a one-sided inequality residual is stated.}
\label{tab:numerical-crosschecks}
\footnotesize
\setlength{\tabcolsep}{5pt}
\begin{tabular}{@{}P{0.18\textwidth}P{0.34\textwidth}P{0.39\textwidth}@{}}
\toprule
Mathematical object & Independent computations & Coverage and discrepancy \\
\midrule
Resonant spectrum and local rate
& Direct matrix spectrum, congruence-filtered Bessel spectrum, stable Fourier
gap, and nonlinear phase flow
& 144 spectra and 12 rate fits; full-spectrum error
$1.554\times10^{-15}$, Bessel--Fourier gap error
$6.051\times10^{-15}$, and relative rate error
$1.059\times10^{-7}$ \\
\addlinespace
Geometric contraction bounds
& Direct sphere/phase vector fields against the Dini inequalities and
integrated trajectories against the comparison solutions
& 7057 derivative, boundary, and trajectory evaluations; minimum inequality
residual $-3.886\times10^{-16}$ and maximum trajectory excess
$2.961\times10^{-10}$ \\
\addlinespace
Twisted states and drift
& Direct right-hand side, closed drift formula, analytic Jacobian, and
centered finite-difference Jacobian
& 70 equilibrium and 50 nonresonant configurations; drift-formula error
$3.764\times10^{-16}$ and Jacobian error $3.070\times10^{-9}$ \\
\addlinespace
Multi-frequency gap
& Symmetric similarity transform, normalized Laplacian, generalized
eigenproblem, and nonlinear sphere flow
& Nine three-way gap comparisons, 320 positive-semidefiniteness stress
configurations, and five rate fits; gap disagreement
$1.527\times10^{-15}$ and relative rate error
$2.805\times10^{-6}$ \\
\bottomrule
\end{tabular}
\end{table}

For the resonant calculation, structural zero modes are determined by the
exact $\gcd(m,n)$ reachability mask rather than by counting all
floating-point zeros.  This distinction matters at large $\beta$, where
truncating a Bessel tail can numerically erase an analytically positive
reachable mode.  The twelve nonlinear fits span invisible, coprime,
non-coprime, and effective-period-two cases, and their fitted decay rates
agree with the tangent spectral gaps at the accuracy reported in
\cref{tab:numerical-crosschecks}.

The geometric checks include orthogonal initial data, exact bipodal
stationary states, all-pairs-active equiangular configurations, two
orientations relative to the RoPE planes, and inverse temperatures up to
$\beta=16$.  The small negative Dini residual in the table is at the scale of
double-precision roundoff, and the apparent trajectory excess is well below
the $2\times10^{-8}$ integration-comparison tolerance.  The twisted-state
calculation separately treats generic resonance, even resonant
anti-alignment, and the odd antipodal family; nonresonant rows compare the direct drift with
\eqref{eq:twisted-drift-closed-form} and its frequency-dependent
$O_\omega(n^{-1})$ bound.

For the two-plane kernel, the three spectral formulations agree to near
machine precision, while trajectories initialized along the slow tangent mode
recover the predicted gap.  Halving the sphere-integrator step from
$2\times10^{-2}$ to $10^{-2}$ changes the fitted rate by
$1.173\times10^{-8}$ when normalized by the predicted gap.  On an independent circle problem,
stage-wise retracted RK4 agrees with a high-accuracy phase-coordinate DOP853
solution, preserves unit norms to roundoff, and exhibits fourth-order
step-halving behavior.

Finally, the natural-energy derivative is evaluated both from its analytic
directional-derivative identity and by centered differences.  The two exact
same-system states in \cref{prop:energy-no-go} reproduce opposite signs
independently of any parameter scan.  Similarly, the normalized residual
update \eqref{eq:normalized-residual-step} converges at first order to the
sphere vector field \eqref{eq:sphere-flow}, providing a direct numerical
check of the discrete-to-continuous convention.

\section{Limitations and open problems}
\label{sec:limitations}

The model isolates one mechanism: RoPE rotates queries and keys, the value
matrix is the identity, and normalization constrains every token to a sphere.
It omits learned query/key/value maps, feed-forward layers, multiple heads,
causal masking, stochastic depth, and finite-step or finite-depth effects
beyond the first-order normalized-residual limit
\eqref{eq:normalized-residual-limit}.  The theorems therefore explain a
controlled dynamical limit rather than the full behavior of a trained
Transformer.  Architectures that rotate values or replace their coupling law,
such as RoVE and oscillator-inspired attention
\citep{garciacastellanos2026rove,nunley2026kuramoto}, lie outside this model.

The global results are kernel-generic but regional.  A fixed closed hemisphere is invariant but
does not itself imply consensus, and the explicit contraction theorems require
pairwise non-obtuse data or, on the circle, a strict open semicircle.  Bipodal
stationary states show that this boundary cannot be erased by a continuity
argument.

Two asymptotic questions remain outside the theorem set.  First, the
dense-phase expression
\begin{equation*}
1-\frac{I_1(\beta)}{I_0(\beta)}\sim\frac{1}{2\beta}
\end{equation*}
is expected when $\beta,L\to\infty$ and
$L/\sqrt\beta\to\infty$, but the required uniform aliasing error bound is not
yet proved.  Second, a quantitative nonresonant finite-chain limit requires
uniform boundary and Diophantine control.  These statements remain a partial
claim and a conjecture, respectively, and are not used by any closed theorem.

Finally, the generic resonant twisted family has a large center subspace.
Linear analysis proves instability but does not determine the center dynamics
or justify a claim that these equilibria organize nonlinear transients.  Such a
claim would require a center-manifold analysis or dedicated trajectory evidence.

\section{Conclusion}
\label{sec:conclusion}

RoPE separates the score geometry from the value dynamics in spherical
self-attention.  It can break the natural vanilla energy law, but leaves enough
reversible kernel structure to compute local consensus rates exactly and enough
uniform positivity to prove contraction in explicit geometric regions.  The
resulting picture is not a universal ``RoPE prevents collapse'' statement:
consensus remains locally stable and regionally attracting, while its rate can
be exponentially small on a fixed resonant ring.  RoPE selects a
score-flattening twisted branch and an antipodal exception within the broader
landscape of polygonal attention equilibria, while multiple
frequencies make the local rate depend on finite sampling and consensus-plane
energies in a way that can be non-monotone.  Explicit theorem boundaries and
independent numerical cross-checks keep these conclusions separate from the
unresolved dense-phase and nonresonant-chain limits.

\bibliographystyle{plainnat}
\bibliography{references}

\appendix
\numberwithin{equation}{section}
\renewcommand{\theequation}{\thesection-\arabic{equation}}
\section{Structural proofs}
\label{app:structural-proofs}

\subsection{Well-posedness, detailed balance, and the sharp floor}

\begin{proof}[Proof of \cref{thm:wellposed-balance-floor}]
The right-hand side of \eqref{eq:sphere-flow} is smooth on a neighborhood of
$(\Sphere^{d-1})^n$: all scores and exponentials are smooth, and every
denominator $Z_i$ is strictly positive.  Moreover,
$\ip{x_i}{\dot x_i}=0$, so the product sphere is invariant.  Local existence
and uniqueness therefore follow from the standard ODE theorem, and compactness
of $(\Sphere^{d-1})^n$ rules out finite-time escape, giving a unique global
solution.

Orthogonality of the rotations gives
\[
s_{ji}=\ip{R_{p_j}x_j}{R_{p_i}x_i}=s_{ij},
\]
hence $W=W^\top$.  Row stochasticity is immediate from the definition of
$A$.  If $Z_\Sigma=\sum_k Z_k$, then
\[
\pi_iA_{ij}=\frac{Z_i}{Z_\Sigma}\frac{W_{ij}}{Z_i}
=\frac{W_{ij}}{Z_\Sigma}
=\frac{W_{ji}}{Z_\Sigma}=\pi_jA_{ji},
\]
which proves \eqref{eq:detailed-balance}.

Every score lies in $[-1,1]$.  For fixed $i,j$,
\[
A_{ij}
=\frac{1}{1+\sum_{k\ne j}\exp\!\bigl(\beta(s_{ik}-s_{ij})\bigr)}
\ge \frac{1}{1+(n-1)\e^{2\beta}},
\]
because $s_{ik}-s_{ij}\le2$.  To see sharpness over the model class, fix a
unit vector $u$ and distinct indices $i,j$, and choose
\[
x_k=R_{p_k}^{\top}u\quad(k\ne j),
\qquad x_j=-R_{p_j}^{\top}u.
\]
Then row $i$ has $s_{ij}=-1$ and $s_{ik}=1$ for every $k\ne j$, so equality
holds in \eqref{eq:sharp-floor}.  The comparison with
$\e^{-2\beta}/n$ follows by direct algebra, with equality only at
$\beta=0$.
\end{proof}

\subsection{The natural energy has no uniform sign law}

For the circle model, put
\[
W_{ij}=\e^{\beta\cos\delta_{ij}},\qquad
u_i=\sum_jW_{ij}\sin\delta_{ij},\qquad
v_i=\sum_jW_{ij}\sin(\theta_i-\theta_j).
\]
Since the ordered-pair energy counts both $(i,j)$ and $(j,i)$,
\begin{equation}
\frac{\partial E_+}{\partial\theta_i}=-2\beta u_i,
\qquad
\dot\theta_i=-\frac{v_i}{Z_i},
\qquad
\dot E_+=2\beta\sum_i\frac{u_iv_i}{Z_i}.
\label{eq:energy-derivative-identity}
\end{equation}

\begin{proof}[Proof of \cref{prop:energy-no-go}]
Substitution of $(p_1,p_2,p_3)=(0,1,2)$, $\beta=1$, and $\omega=\pi$
into \eqref{eq:energy-derivative-identity} gives, at the first state,
\begin{equation}
\left.\dot E_+\right|_{(0,\pi/6,\pi/3)}
=
\frac{-1+3\e^{1+\sqrt3}}
{\e^{\sqrt3/2}+\e^{1/2+\sqrt3}+\e^{1+\sqrt3}}
\approx 1.666364304459315>0.
\label{eq:positive-energy-witness}
\end{equation}
At the second state, the same identity yields
\begin{equation}
\left.\dot E_+\right|_{(0,\pi/6,0)}
=-
\frac{5\e^{1+\sqrt3/2}+4}
{\e^{\sqrt3/2}(1+2\e^{1+\sqrt3/2})(2+\e^{1+\sqrt3/2})}
\approx -0.1296131359239394<0.
\label{eq:negative-energy-witness}
\end{equation}
Thus one fixed admissible system contains states with both derivative signs.
This proves the stated no-go for $E_+$ and $-E_+$, without making a claim
about other candidate Lyapunov functions.
\end{proof}

\subsection{Consensus and twisted configurations}

\begin{proof}[Proof of \cref{prop:consensus-kernel}]
At consensus, the value average in each row equals $x^\star$, so its tangent
projection vanishes.  On rotation plane $\ell$,
\[
\ip{R_{p_i}x^\star_{(\ell)}}{R_{p_j}x^\star_{(\ell)}}
=\norm{x^\star_{(\ell)}}^2
\cos\!\bigl(\omega_\ell(p_i-p_j)\bigr).
\]
Summing over the mutually orthogonal planes proves
\eqref{eq:consensus-kernel}, including its dependence only on the energies
$a_\ell$.
\end{proof}

\begin{proof}[Proof of \cref{prop:twisted-equilibria}]
At \eqref{eq:twisted-state}, every adjusted score phase is zero, hence
$A_{ij}=1/n$ and
\begin{equation}
F_i:=\dot\theta_i
=\frac1n\sum_{j=0}^{n-1}\sin\!\bigl(\omega(i-j)\bigr).
\label{eq:twisted-drift-sum}
\end{equation}
Let $G_n(\omega)=\sum_{j=0}^{n-1}\e^{-\ii\omega j}$.  Then
$nF_i=\operatorname{Im}(\e^{\ii\omega i}G_n)$.  If
$\omega\notin2\pi\Z$, the geometric-sum formula gives
\[
G_n(\omega)
=\e^{-\ii\omega(n-1)/2}
\frac{\sin(n\omega/2)}{\sin(\omega/2)},
\]
and therefore
\begin{equation}
F_i=
\frac{\sin(n\omega/2)}{n\sin(\omega/2)}
\sin\!\left(\omega\left(i-\frac{n-1}{2}\right)\right).
\label{eq:twisted-drift-closed-form}
\end{equation}
This proves the two bounds in \eqref{eq:twisted-drift}.

If $\omega\in\pi\Z$, every summand in
\eqref{eq:twisted-drift-sum} vanishes.  Suppose instead that
$\sin\omega\ne0$ and $F_i=0$ for all $i$.  The equations for two consecutive
indices are two real linear constraints on the real and imaginary parts of
$G_n$; their determinant is $\sin\omega\ne0$.  Hence $G_n=0$, which is
equivalent to $n\omega\in2\pi\Z$.  Conversely, $G_n=0$ makes every $F_i$
zero.  This proves the equivalence \eqref{eq:twisted-iff}.

The frequency dependence in the final estimate is essential.  Taking
$\omega_n=\pi/n$ and $i=0$ gives
\[
F_0=-\frac1n\sum_{j=0}^{n-1}\sin\frac{\pi j}{n}
\longrightarrow-\frac2\pi,
\]
so no $O(n^{-1})$ bound can hold uniformly for frequencies approaching
$2\pi\Z$ with $n$.
\end{proof}

\subsection{Consensus linearization}

\begin{proof}[Proof of \cref{thm:consensus-linearization}]
Write
\[
y_i(x)=\sum_jA_{ij}(x)x_j,
\qquad F_i(x)=\proj{x_i}y_i(x).
\]
For tangent perturbations $\xi_i\perp x^\star$, the first variation at
consensus is
\[
\delta y_i
=\sum_j(\delta A_{ij})x^\star+
  \sum_jA_{ij}^\star\xi_j.
\]
Row stochasticity implies $\sum_j\delta A_{ij}=0$, so
$\delta y_i=\sum_jA_{ij}^\star\xi_j$, which is tangent at $x^\star$.
Differentiating $P_xy=y-\ip{x}{y}x$ at $x=y=x^\star$ now gives
\[
\delta(P_xy)=\delta y_i-\xi_i.
\]
This proves \eqref{eq:tangent-linearization} and
\eqref{eq:kronecker-jacobian}.

Detailed balance implies $(\pi^\star)^\top A^\star=(\pi^\star)^\top$.
Consequently, along the fixed linearized system,
\[
\frac{\mathrm d}{\mathrm dt}\sum_i\pi_i^\star\xi_i
=\sum_j\left(\sum_i\pi_i^\star A_{ij}^\star\right)\xi_j
-\sum_i\pi_i^\star\xi_i=0.
\]
The weights here are frozen at the equilibrium, which is why this calculation
does not yield a nonlinear conservation law.
\end{proof}

\section{Spectral and twisted-state proofs}
\label{app:spectral-proofs}

\subsection{Bessel aliasing on a resonant ring}

\begin{proof}[Proof of \cref{thm:bessel-spectrum}]
At consensus and under \eqref{eq:resonance-data},
\[
W^\star_{ij}=e^{\beta\cos(2\pi m(i-j)/n)}.
\]
Thus $W^\star$ is circulant and has a constant row sum $Z$; consequently
$A^\star=W^\star/Z$ is circulant, symmetric, and doubly stochastic.  Use the
absolutely convergent expansion
\begin{equation}
\e^{\beta\cos\vartheta}
=\sum_{k\in\Z}I_k(\beta)\e^{\ii k\vartheta}.
\label{eq:bessel-generating-function}
\end{equation}
Substituting \eqref{eq:bessel-generating-function}, the unnormalized
eigenvalue at Fourier mode $q$ is
\begin{align*}
\widehat W_q
&=\sum_{h=0}^{n-1}
  \e^{\beta\cos(2\pi mh/n)}\e^{-2\pi\ii qh/n}\\
&=\sum_{k\in\Z}I_k(\beta)
  \sum_{h=0}^{n-1}\e^{2\pi\ii(mk-q)h/n}\\
&=n\sum_{\substack{k\in\Z\\mk\equiv q\ (\mathrm{mod}\ n)}}I_k(\beta).
\end{align*}
Since $Z=\widehat W_0$, division proves
\eqref{eq:bessel-alias-spectrum}.

The congruence $mk\equiv q\pmod n$ is solvable exactly when
$g=\gcd(m,n)$ divides $q$.  After writing $m=gm'$ and $q=gs$, the reachable
solutions form one residue class modulo $L=n/g$, because $m'$ is invertible
modulo $L$.  For $\beta>0$, every $I_k(\beta)$ is strictly positive, so the
$L$ reachable Fourier modes have positive eigenvalues and the remaining
$n-L$ modes have eigenvalue zero.  At $\beta=0$, only $I_0(0)=1$ survives,
leaving the Perron eigenvalue one and $n-1$ zeros.
\end{proof}

\subsection{Positive spectrum and the exact equality case}

\begin{proof}[Proof of \cref{thm:psd-slowdown}]
For each sampled position define
\[
\phi_i=
\bigl(
\sqrt{a_\ell}\cos(\omega_\ell p_i),
\sqrt{a_\ell}\sin(\omega_\ell p_i)
\bigr)_{\ell=1}^{r}.
\]
Then $K_{a}(p_i-p_j)=\ip{\phi_i}{\phi_j}$, so the matrix $K=[K_{ij}]$
is positive semidefinite.  By the Schur product theorem every Hadamard power
$K^{\circ k}$ is positive semidefinite, and hence
\begin{equation}
W^\star=\exp_{\circ}(\beta K)
=\sum_{k=0}^{\infty}\frac{\beta^k}{k!}K^{\circ k}\succeq0.
\label{eq:entrywise-exponential-psd}
\end{equation}
With $D=\diag(W^\star\one)$,
\[
D^{1/2}A^\star D^{-1/2}=D^{-1/2}W^\star D^{-1/2}=:S\succeq0.
\]
Thus $A^\star$ is similar to a symmetric positive semidefinite matrix.  Its
entries are strictly positive and it is stochastic, so Perron--Frobenius gives
a simple eigenvalue one and all remaining eigenvalues strictly below one.
This proves \eqref{eq:positive-spectrum}.

At $\beta=0$, $W^\star=\one\one^\top$, so $\gamma=1$.  Let
$\beta>0$.  Because the spectrum is nonnegative, $\gamma=1$ exactly when
$W^\star$ has rank one.  Its diagonal is the constant $\e^\beta$.  Vanishing
of every $2\times2$ principal minor then gives
\[
0=\e^{2\beta}-\e^{2\beta K_{ij}},
\qquad\text{hence}\qquad K_{ij}=1
\quad\text{for all }i,j.
\]
The converse is immediate.  Finally,
\[
1-K_{ij}=\sum_\ell a_\ell
\left[1-\cos\!\bigl(\omega_\ell(p_i-p_j)\bigr)\right].
\]
All summands are nonnegative, so $K_{ij}=1$ for all pairs precisely when
every active plane satisfies \eqref{eq:invisible-planes}.
\end{proof}

\subsection{Fixed-period low-temperature asymptotics}

\begin{proof}[Proof of \cref{thm:finite-ring-asymptotics}]
The positive spectrum from \cref{thm:bessel-spectrum} is, up to a permutation,
the spectrum of the $L$-point phase grid.  Scaling all entries of its kernel by
$\e^{-\beta}$ does not alter the normalized matrix.  Put
\[
\Delta_k=1-\cos\frac{2\pi k}{L},
\qquad b_k=\e^{-\beta\Delta_k}.
\]
For $L\ge2$, the exact gap is
\begin{equation}
\gamma_L(\beta)=
\min_{1\le s\le L-1}
\frac{\displaystyle
\sum_{k=0}^{L-1}b_k
\left(1-\cos\frac{2\pi sk}{L}\right)}
{\displaystyle\sum_{k=0}^{L-1}b_k}.
\label{eq:exact-finite-ring-gap}
\end{equation}
If $L=1$, the kernel has rank one and $\gamma=1$.  If $L=2$, the scaled
weights are $b_0=1$ and $b_1=\e^{-2\beta}$, so the only nontrivial positive
eigenvalue is
\[
\lambda_1=\frac{1-\e^{-2\beta}}{1+\e^{-2\beta}}=\tanh\beta,
\]
which proves \eqref{eq:antialigned-gap}.

Now fix $L\ge3$.  The two nearest nonzero phase points, $k=1$ and $k=L-1$,
have the common deficit $\Delta_L=1-\cos(2\pi/L)$; every other nonzero
point has strictly larger deficit.  Therefore
\[
\sum_{k=0}^{L-1}b_k
=1+2\e^{-\beta\Delta_L}+o(\e^{-\beta\Delta_L}),
\]
and, for each $s=1,\ldots,L-1$,
\[
\sum_{k=0}^{L-1}b_k
\left(1-\cos\frac{2\pi sk}{L}\right)
=2\left(1-\cos\frac{2\pi s}{L}\right)
\e^{-\beta\Delta_L}+o(\e^{-\beta\Delta_L}).
\]
Because the set of $s$ is finite, the remainder is uniform over it, and
\[
\min_{1\le s\le L-1}
\left(1-\cos\frac{2\pi s}{L}\right)=\Delta_L.
\]
Substitution into \eqref{eq:exact-finite-ring-gap} proves
\eqref{eq:fixed-ring-asymptotic}.  The proof is pointwise in fixed $L$ and
therefore supplies no uniform error estimate for the joint limit described
after the theorem.
\end{proof}

\subsection{Linear spectra of twisted configurations}

At a twisted configuration, the first variation of every attention score
vanishes because the adjusted phase is zero and the derivative of cosine at
zero is zero.  Differentiating the value torque therefore gives, for a phase
perturbation $\eta$,
\begin{equation}
(J\eta)_i
=\frac1n\sum_j C_{ij}(\eta_j-\eta_i),
\qquad C_{ij}=\cos\!\bigl(\omega(i-j)\bigr).
\label{eq:general-twisted-jacobian}
\end{equation}

\begin{proof}[Proof of \cref{prop:twisted-instability}]
For a nontrivial resonant frequency, the root-of-unity sum gives
$C\one=0$, so \eqref{eq:general-twisted-jacobian} reduces to
\eqref{eq:twisted-jacobian}.  Let
$z_j=\e^{2\pi\ii mj/n}$.  If $2\bar m\not\equiv0\pmod n$, the Fourier
vectors $z$ and $\bar z$ are orthogonal, and
\[
\frac1n C=\frac12\left(
\frac{z}{\sqrt n}\frac{z^*}{\sqrt n}
+\frac{\bar z}{\sqrt n}\frac{\bar z^*}{\sqrt n}
\right).
\]
This proves \eqref{eq:generic-twisted-spectrum}.  If
$\bar m=n/2$, then $C_{ij}=(-1)^{i-j}$ has rank one and eigenvalue $n$,
which proves \eqref{eq:even-antipodal-spectrum}.  In both cases
$J\one=0$, as also follows from global phase equivariance.
\end{proof}

\begin{proof}[Proof of \cref{prop:odd-antipodal-saddle}]
Write $n=2r+1$ and $s_j=(-1)^j$ for $j=0,\ldots,n-1$, so
$\one^\top s=1$.  Then $C=ss^\top$, $C\one=s$, and
\eqref{eq:general-twisted-jacobian} becomes
\begin{equation}
J=\frac1n\bigl(ss^\top-\diag(s)\bigr).
\label{eq:odd-antipodal-jacobian}
\end{equation}
By \eqref{eq:odd-antipodal-jacobian}, the indices with $s_j=1$ number $r+1$.
Vectors supported on them and with
coordinate sum zero form an $r$-dimensional eigenspace of eigenvalue $-1/n$.
The $r$ indices with $s_j=-1$ analogously supply an $(r-1)$-dimensional
eigenspace of eigenvalue $1/n$.  Finally,
\[
J\one=0,
\qquad
J\left(s-\frac1n\one\right)=s-\frac1n\one.
\]
These invariant subspaces have total dimension $n$ and prove
\eqref{eq:odd-antipodal-spectrum}.  Removing the unique rotation mode leaves
no zero eigenvalue and leaves $r$ stable and $r$ unstable directions.
\end{proof}

\subsection{Multi-frequency geometry and analytic counterexamples}

\begin{proof}[Proof of \cref{thm:multifrequency-spectrum}]
The Gram--Schur argument in the proof of \cref{thm:psd-slowdown} applies
verbatim to $K_a$, proving $W\succeq0$, the stated similarity, and the
eigenvalue ordering.  The calculation in the proof of
\cref{thm:consensus-linearization} gives the Kronecker tangent Jacobian.
Finally, the rank-one argument following
\eqref{eq:entrywise-exponential-psd} proves the equality condition
\eqref{eq:invisible-planes}.
\end{proof}

\begin{proof}[Proof of \cref{prop:analytic-energy-nonmonotonicity}]
For lag $h$, the two-plane kernel is
\[
K_a(h)=a\cos(\pi h/2)+(1-a)\cos(\pi h).
\]
Hence $K_a(1)=a-1$ and $K_a(2)=1-2a$.  After scaling all weights by
$\e^{-\beta}$, set
\[
x=\e^{\beta(a-2)},\qquad y=\e^{-2\beta a}.
\]
The kernel and its degree vector are
\[
W=\begin{pmatrix}1&x&y\\x&1&x\\y&x&1\end{pmatrix},
\qquad
W\one=(1+x+y,\,1+2x,\,1+x+y)^\top.
\]
The reversal-odd vector $(1,0,-1)^\top$ gives the first non-Perron
eigenvalue in \eqref{eq:three-token-eigenvalues}.  The second follows from
$\operatorname{tr}(D^{-1}W)=1+\lambda_-+\lambda_+$, proving that formula.

For $\beta=\tfrac34\log2$, direct substitution at $a=0,2/3,1$ gives
exactly \eqref{eq:analytic-nonmonotone-values}.  The three values are
$0.564435\ldots$, $0.75$, and $0.668175\ldots$, respectively, which proves
the strict interior value above both endpoints claimed in the proposition.
\end{proof}

\begin{proof}[Proof of \cref{prop:analytic-frequency-reversal}]
After scaling by the common diagonal weight, the two-token kernel is
\[
W=\begin{pmatrix}1&q\\q&1\end{pmatrix},
\qquad q=\e^{\beta[\cos(\omega\Delta)-1]}.
\]
Its non-Perron eigenvalue is $(1-q)/(1+q)$, proving
\eqref{eq:two-token-gap}.  For $\beta>0$ and $\Delta=1$, the low and high
frequencies give $q_{\rm low}=\e^{-\beta}$ and
$q_{\rm high}=\e^{-2\beta}$; since the gap is increasing in $q$, the high
frequency is slower.  For $\Delta=2$, they give
$q_{\rm low}=\e^{-2\beta}$ and $q_{\rm high}=1$, so the high-frequency gap
is one and the ordering reverses.
\end{proof}

\section{Proofs of the global geometric estimates}
\label{app:global-proofs}

All functions below are evaluated along the unique solution from
\cref{thm:wellposed-balance-floor}.  We repeatedly use the uniform lower bound
$A_{ij}\ge\abar>0$.

\begin{lemma}[Finite active-set calculus]
\label{lem:active-set-calculus}
Let $f_1,\ldots,f_N$ be continuously differentiable functions and set
$f_{\min}=\min_\alpha f_\alpha$ and
$f_{\max}=\max_\alpha f_\alpha$.  If $I_{\min}$ and $I_{\max}$ are the
sets attaining the respective extrema at time $t$, then
\begin{equation}
\Dini f_{\min}(t)=\min_{\alpha\in I_{\min}(t)}\dot f_\alpha(t),
\qquad
\Dini f_{\max}(t)=\max_{\alpha\in I_{\max}(t)}\dot f_\alpha(t).
\label{eq:active-set-formulas}
\end{equation}
In particular, if
$D=\max_i\theta_i-\min_i\theta_i$, then
\begin{equation}
\Dini D
=\max_{i\in I_+}\dot\theta_i-\min_{k\in I_-}\dot\theta_k.
\label{eq:diameter-active-set}
\end{equation}
\end{lemma}

\begin{proof}
Because the family is finite, the expansions
$f_\alpha(t+h)=f_\alpha(t)+h\dot f_\alpha(t)+o(h)$ are uniform in
$\alpha$.  Every inactive index has a positive gap from the extremal value and
remains inactive for all sufficiently small $h>0$.  Taking the minimum or
maximum of the remaining first-order expansions proves
\eqref{eq:active-set-formulas}; subtraction gives
\eqref{eq:diameter-active-set}.
\end{proof}

\subsection{Fixed hemispheres}

\begin{proof}[Proof of \cref{prop:hemisphere-invariance}]
Let $m=m_w$ and take any active minimizing index $i$.  Differentiating
$h_i=\ip{x_i}{w}$ gives
\begin{equation}
\dot h_i=\sum_jA_{ij}
\left(h_j-\ip{x_i}{x_j}m\right).
\label{eq:hemisphere-derivative}
\end{equation}
Since $h_j\ge m$,
\[
h_j-\ip{x_i}{x_j}m
\ge m\bigl(1-\ip{x_i}{x_j}\bigr).
\]
If $m\ge0$, the right-hand side is nonnegative.  If $m<0$, then
$0\le1-\ip{x_i}{x_j}\le2$, so it is at least $2m$.  The active-set lemma
therefore gives the global inequality
\begin{equation}
\Dini m\ge2\min\{m,0\}.
\label{eq:hemisphere-global-dini}
\end{equation}
The function $m$ is absolutely continuous.  For its negative part
$m_-:=\max\{-m,0\}$, \eqref{eq:hemisphere-global-dini} implies, almost
everywhere, $\dot m_-\le2m_-$.  If $m(0)\ge0$, Gr\"onwall's inequality
forces $m_-(t)=0$ for all time.  Once $m\ge0$, the active derivatives in
\eqref{eq:hemisphere-derivative} are nonnegative, so $m$ is nondecreasing and
$m(t)\ge m(0)$.
\end{proof}

\subsection{Pairwise non-obtuse data}

For $g_{ij}=\ip{x_i}{x_j}$, direct differentiation of
\eqref{eq:sphere-flow} gives
\begin{equation}
\dot g_{ij}
=\sum_kA_{ik}(g_{kj}-g_{ik}g_{ij})
+\sum_kA_{jk}(g_{ik}-g_{jk}g_{ij}).
\label{eq:pairwise-inner-product-derivative}
\end{equation}

\begin{proof}[Proof of \cref{thm:acute-contraction}]
Let $(i,j)$ be any active pair with $g_{ij}=c$.  First suppose $c<0$.
Using $g_{kj}\ge c$, $g_{ik}\le1$, and the analogous inequalities in the
second sum,
\[
g_{kj}-cg_{ik}\ge c(1-g_{ik})\ge2c,
\qquad
g_{ik}-cg_{jk}\ge c(1-g_{jk})\ge2c.
\]
Thus every active pair obeys $\dot g_{ij}\ge4c$, and
$\Dini c\ge4c$ whenever $c<0$.  At $c=0$, every summand in
\eqref{eq:pairwise-inner-product-derivative} for an active pair is
nonnegative.  Applying Gr\"onwall to $c_-:=\max\{-c,0\}$ therefore proves
that $c(0)\ge0$ implies $c(t)\ge0$ for all $t$.

Now assume $c\ge0$.  For every $k$,
\[
g_{kj}-cg_{ik}\ge c(1-g_{ik})\ge0,
\qquad
g_{ik}-cg_{jk}\ge c(1-g_{jk})\ge0.
\]
Retaining only $k=j$ in the first sum of
\eqref{eq:pairwise-inner-product-derivative} and $k=i$ in the second gives
\[
\dot g_{ij}\ge(A_{ij}+A_{ji})(1-c^2)
\ge2\abar(1-c^2).
\]
Taking the minimum over active pairs proves \eqref{eq:acute-dini}.  The scalar
comparison equation also gives $c(t)\ge\tanh(2\abar t)$ when $c(0)=0$.

At almost every time, differentiation of \eqref{eq:rho-def} yields
\[
\dot\rho=-\frac{2\dot c}{(1+c)^2}
\le-4\abar\frac{1-c^2}{(1+c)^2}
=-4\abar\rho.
\]
Gr\"onwall proves \eqref{eq:rho-decay}.  Finally,
\[
\tan^2\frac{D_{\Sphere}}2
=\frac{1-\cos D_{\Sphere}}{1+\cos D_{\Sphere}}
=\frac{1-c}{1+c}=\rho,
\]
which proves \eqref{eq:acute-half-angle}.
\end{proof}

\subsection{Strict semicircles and the sharp boundary}

\begin{proof}[Proof of \cref{thm:semicircle-contraction}]
As long as $D<\pi$, take an active maximum $i$ and active minimum $k$.
Every phase difference $\theta_j-\theta_i$ lies in $[-D,0]$, so all its
sines are nonpositive and
\[
\dot\theta_i\le-A_{ik}\sin D.
\]
Similarly, every $\theta_j-\theta_k$ lies in $[0,D]$, giving
\[
\dot\theta_k\ge A_{ki}\sin D.
\]
The active-set identity \eqref{eq:diameter-active-set} and the attention floor
now imply
\[
\Dini D\le-(A_{ik}+A_{ki})\sin D\le-2\abar\sin D.
\]
On the maximal interval where $D<\pi$, this makes $D$ nonincreasing, so
$D(t)\le D_0<\pi$ and no finite exit can occur.  Thus the estimate holds
globally.  For $q=\tan(D/2)$, almost-everywhere differentiation gives
\[
\dot q
\le-\abar\sec^2(D/2)\sin D=-2\abar q.
\]
One final application of Gr\"onwall proves \eqref{eq:semicircle-bound}.
\end{proof}

\begin{proof}[Proof of \cref{prop:bipodal-boundary}]
For $x_i=\sigma_i u$, every row average is
\[
\sum_jA_{ij}x_j=\left(\sum_jA_{ij}\sigma_j\right)u,
\]
which is parallel to $x_i$.  Its tangent projection at $x_i$ vanishes.  When
$d=2$ and both signs occur, the two phases differ by $\pi$, proving the
boundary assertion.
\end{proof}

\subsection{Convergence to one point and observable rates}

\begin{proof}[Proof of \cref{cor:single-point-tail}]
Both contraction theorems give
$q(t):=\tan(D(t)/2)\le Q(t)$.  The maximal chordal distance is therefore
\begin{equation}
L(t)=2\sin\frac{D(t)}2
=\frac{2q(t)}{\sqrt{1+q(t)^2}}
\le\frac{2Q(t)}{\sqrt{1+Q(t)^2}}.
\label{eq:chord-bound-for-tail}
\end{equation}
Using row stochasticity and the operator norm of an orthogonal projection,
\[
\norm{\dot x_i}
\le\sum_jA_{ij}\norm{\proj{x_i}(x_j-x_i)}
\le\sum_jA_{ij}\norm{x_j-x_i}
\le L(t).
\]
For $s>t$, the geodesic distance is bounded by the Riemannian length of the
trajectory segment, so \eqref{eq:chord-bound-for-tail} gives
\begin{align*}
\distS(x_i(t),x_i(s))
&\le\int_t^s\norm{\dot x_i(r)}\,\mathrm dr\\
&\le\frac1{\abar}
\left[\operatorname{arsinh}Q(t)-\operatorname{arsinh}Q(s)\right].
\end{align*}
Thus each trajectory is Cauchy.  The sphere is complete, and $D(t)\to0$, so
all token limits coincide at some $x_\infty$.  Sending $s\to\infty$ and using
$\operatorname{arsinh}z\le z$ for $z\ge0$ proves
\eqref{eq:geodesic-tail}.

Finally, $D=2\arctan q\le2q$, $L=2q/\sqrt{1+q^2}\le2q$, and
$Y=1-c=L^2/2$.  Substitution of $q(t)\le q_0\e^{-2\abar t}$ proves all
three estimates in \eqref{eq:observable-rates} and explains the factor two
between the logarithmic rates of $Y$ and $L$.
\end{proof}

\section{Numerical methods and parameter grids}
\label{app:numerical-protocol}

The calculations independently cross-check the analytic statements and do not
supply proof steps for any theorem.  All computations use double precision,
literal finite-chain position differences, and a periodic representation only
for the exactly resonant circulant problem.

\subsection{Error conventions and time integration}

For spectra, the reported error is the largest absolute discrepancy between
the compared eigenvalues after sorting, or the absolute discrepancy between
two gap formulas.  Structural zeros are instead identified from exact
congruence reachability.  For a measured decay rate $\widehat\gamma$ and a
predicted gap $\gamma>0$, the relative error is
\[
\frac{\abs{\widehat\gamma-\gamma}}{\gamma}.
\]
One-sided Dini residuals are signed so that the proved inequality corresponds
to a nonnegative value.

Circle trajectories are integrated in phase coordinates with DOP853.  The
local-rate fits start from a perturbation of amplitude $10^{-4}$ in the slow
mode, use 360 output times, and fit $\log$ transverse norm over
$0.75\le\gamma t\le3.75$ with relative and absolute solver tolerances
$10^{-10}$ and $10^{-12}$.  Sphere trajectories use classical RK4 with every
stage and accepted state retracted row-wise to the sphere.  The
multi-frequency rate calculation uses step size $10^{-2}$, the same
perturbation amplitude, and the same scaled-time fitting window.  A separate
order check compares step sizes $0.1$, $0.05$, and $0.025$ with a
phase-coordinate DOP853 reference computed at tolerances $10^{-13}$ and
$10^{-15}$; both successive error ratios exceed 14, consistent with fourth
order.

\subsection{Resonant spectra and local rates}

The resonant grid uses $n\in\{8,16,32,64,128\}$,
$\beta\in\{0.5,1,2,4,8,16\}$, and the distinct modes in
$\{0,1,2,n/4,n/2\}$.  For each of the resulting 144 configurations, the
complete spectrum from direct diagonalization of the normalized kernel is
compared with the congruence-filtered Bessel sum
\eqref{eq:bessel-alias-spectrum}; the gap is also evaluated with a
cancellation-free Fourier sum.  Scaled modified Bessel functions prevent
overflow.  The exact unreachable-mode mask is determined by $\gcd(m,n)$, so
an analytically positive mode below working precision is not misclassified as
a structural zero.

The nonlinear rate schedule consists of the triples $(n,m,\beta)$
$(8,0,2)$, $(8,1,2)$, $(8,2,2)$, $(8,4,2)$, $(16,1,1)$, $(16,4,4)$,
$(16,8,4)$, $(32,2,2)$, $(32,8,2)$, $(32,16,2)$, $(64,16,1)$, and
$(128,64,1)$.
This schedule covers invisible, coprime, non-coprime, anti-aligned, and
multiple-size regimes without treating a dense parameter scan as statistical
evidence.

\subsection{Geometric-region checks}

The pairwise non-obtuse and closed-hemisphere calculations use
$n\in\{2,5,10\}$, $d\in\{2,4,8\}$, and
$\beta\in\{0,0.5,2,4,8\}$, with 40 fixed seeded configurations per setting
and the standard RoPE
frequencies $10000^{-2\ell/d}$.  Acute states are formed around a random unit
anchor with tangent perturbation magnitudes in $[0.05,0.45]$.
Hemisphere states use margins in $[0.02,0.8]$, with one token placed exactly
on the boundary in half of the configurations.

To exercise nonsmooth active sets, equiangular states use
$n\in\{3,5,7\}$ in two orientations relative to the RoPE planes.  The grid
combines $c_0\in\{0,10^{-12},10^{-6},0.1,0.5\}$ with
$\beta\in\{0,2,8,16\}$.  Every off-diagonal pair is active in these 120
configurations.  The circle Dini grid uses
$n\in\{2,5,10\}$, the same six inverse temperatures as the resonant grid,
$\omega\in\{0,0.7,\pi\}$, initial diameters
$\{\pi/2,0.9\pi,\pi-10^{-6}\}$, and 20 fixed seeded configurations per
setting.  Integrated comparison checks use $n\in\{2,6\}$,
$\beta\in\{0,1,2,4\}$, the same three frequencies, and initial diameters
$\{\pi/2,0.9\pi,\pi-10^{-3}\}$.  Orthogonal and antipodal two-token states,
together with the sharp $\beta=16$ softmax-floor configuration, are evaluated
separately.

\subsection{Twisted states and nonresonant drift}

For each $n=3,\ldots,12$ at $\beta=2$, every nonzero resonant mode is tested
by comparing the direct right-hand side and the analytic Jacobian with a
centered finite-difference Jacobian of step $2\times10^{-7}$.  The odd
antipodal family at $\omega=\pi$ is included separately for odd $n$.
Nonresonant checks use
$\omega\in\{0.37,0.73,1.11,\sqrt2,2.17\}$ and compare the direct drift with
\eqref{eq:twisted-drift-closed-form}, its sharp finite-sum bound, and the
coarser $1/[n\abs{\sin(\omega/2)}]$ envelope.

\subsection{Multi-frequency spectra and nonlinear rates}

The multi-frequency calculations use frequencies $(1,0.01)$.  Three-way gap
comparisons are performed for $n\in\{128,256,384\}$, $\beta=4$, and
high-frequency plane energies $a\in\{0,1/2,1\}$.  Positive
semidefiniteness is evaluated for $n\in\{8,16,32,64\}$ and
$\beta\in\{0,1,4,16\}$ using 20 fixed seeded energy allocations per pair.
Exact $2\pi$ frequency
aliasing, block permutation, and invisibility of a nonzero $2\pi$ frequency
on integer positions are checked separately.  Nonlinear slow-mode rates use
$n=16$, $\beta=4$, and
$a\in\{0,1/4,1/2,3/4,1\}$.

The finite-context observations in
\cref{fig:multifrequency-nonmonotonicity} evaluate endpoint gaps for 18
context lengths between 32 and 629 and five inverse temperatures between
$0.5$ and $8$.  Plane-energy curves use 41 equally spaced allocations for
$n\in\{128,256,384\}$ at $\beta\in\{2,4\}$, together with five allocations
for $n\in\{628,629\}$ at $\beta=4$.  These rows describe finite contexts and
are not counted as independent theorem evidence.

\paragraph{Energy derivative and discrete-time limit.}

The analytic directional derivative
\eqref{eq:energy-derivative-identity} is compared with a centered difference
of the energy along the vector field using step $10^{-6}$; the two agree at
relative and absolute tolerance $2\times10^{-8}$.  The exact states
\eqref{eq:positive-energy-witness}--\eqref{eq:negative-energy-witness} are
evaluated directly and retain opposite signs.

For the normalized residual update
\eqref{eq:normalized-residual-step}, the discrete velocity is compared with
\eqref{eq:sphere-flow} at step sizes $2\times10^{-4}$ and $10^{-4}$ on a
six-token, four-dimensional configuration.  Halving the step reduces the
maximum discrepancy by at least the expected first-order factor, and the
smaller-step discrepancy is below $2\times10^{-4}$.

\end{document}